\documentclass{amsart} 
 \usepackage{amssymb}
\usepackage{amsfonts}
\usepackage{mhequ}
\usepackage{times}
\usepackage{enumerate}

\newcommand{\E}{{\mathbf E}}

\newcommand{\F}{{\mathcal F}}

\newcommand{\LL}{\mathbf {L}}
\newcommand{\Lo}{\mathcal {L}}

\newcommand{\R}{{\mathbf R}}

\def\paral{/\kern-0.55ex/}
\def\parals_#1{/\kern-0.55ex/_{\!#1}}
\def\n#1{|\kern-0.24em|\kern-0.24em|#1|\kern-0.24em|\kern-0.24em|}

\def\eg{\textit{e.g. }}

\def\ie{\textit{i.e. }}

\def\div{\mathop{\rm div}}

\def\s.t.{\mathop {\rm s.t.}}

\newtheorem{theorem}{Theorem}[section]
\newtheorem{lemma}[theorem]{Lemma}

\theoremstyle{definition}

\theoremstyle{remark}
\newtheorem{remark}[theorem]{Remark}

\numberwithin{equation}{section}

\begin{document}

\title{An averaging principle for a completely integrable stochastic Hamiltonian system}
\author{Xue-Mei Li}

\address{Department of Mathematics,
University of Warwick,
Coventry CV4 7AL, U.K.} 
\thanks{Research benefited from a Royal Society Leverhulme Trust Senior Research Fellowship}



\begin{abstract}
 We investigate the effective behaviour of a small transversal perturbation of order $\epsilon$ to a completely integrable stochastic Hamiltonian system, by which we mean a stochastic differential equation whose diffusion vector fields are formed from a completely integrable family of Hamiltonian functions $H_i, i=1,\dots n$. An averaging principle is shown to hold and the action component of the solution converges, as $\epsilon \to 0$, to the solution of a deterministic system of differential equations when the time is rescaled at $1/\epsilon$. An estimate for the rate of the convergence is given. In the case when the perturbation is a Hamiltonian vector field, the limiting deterministic system is constant in which case we show that the action component of the solution scaled at $1/\epsilon^2$ converges  to that of a limiting stochastic differentiable equation.
\end{abstract}
\maketitle

\section{Introduction}

{\it The Model.} A smooth $2n$-dimensional manifold $M$ is said to be a symplectic manifold if it is equipped with a symplectic structure,  that is,  a closed differential two-form $\omega $ which is nondegenerate in the sense that for each $x\in M,$ $\omega(v,w)=0$ for all $w\in T_xM$ implies $v=0$.  Equivalently $M$ admits a set of coordinates mapping such that the coordinate changing maps are symplectic on $\R^{2n}$ with the standard symplectic form $\omega_0=\sum dp_i\wedge d q_i$.

A family of $n$ smooth Hamiltonians $\{H_k\}$ on a $2n$ dimensional symplectic manifold is said to form a (completely) integrable system if they are pointwise Poisson commuting and if the corresponding Hamiltonian vector fields $X_{H_k}$ are linearly independent at almost all points.  Given such an integrable family and a $C^1$ locally  Hamiltonian vector field $V$ commuting with the family of vector fields $X_{H_k}$ in the sense of  $\omega(X_{H_k}, V)=0$,
consider the following model, which we call a {\it completely integrable stochastic symplectic/Hamiltonian system}:
\begin{equation}
\label{1}
dx_t=\sum_{k=1}^n X_{H_k}(x_t)\circ dB_t^k+V(x_t)dt.
\end{equation}
Here $(B_t^k, k=1,\dots, n)$ are pairwise independent Brownian motions on a filtered probability space $(\Omega, \F, \F_t, P)$ with the standard assumptions on the filtration and $\circ$ stands for Stratonovitch integration. 
We have suppressed the chance element $\omega$ here as is conventional. Note that the customary symbol for the symplectic form is unfortunately the same as that for the chance variable, however confusion should not arise as the chance variable will from now on  not be explicitly expressed unless  indicated otherwise. We call respectively $X_{H_k}$ the diffusion vector fields and $V$ the drift vector field for the stochastic differential equation.  

 In the integrable stochastic Hamiltonian system case, the diffusion vector fields span a sub-bundle of the tangent bundle, at least locally. The purpose of the present article is to investigate the effect of a small  perturbation to random systems of this type.  A solution to an integrable Hamiltonian system preserves the energies $H_k$, just as does a solution to any other stochastic Hamiltonian system and there are corresponding invariant manifolds (level sets). The Markov solution restricts to each compact level set and the restriction has generator  $$\Lo_0=\sum_{k=1}^n {1\over 2} \LL_{X_{H_k}} \LL_{X_{H_k}}+\LL_V.$$  
Here $\LL_V$ indicates Lie differentiation in the direction of $V$.
 If the integrable stochastic Hamiltonian system is perturbed by a vector field $\epsilon K$ for $\epsilon>0$ and $K$ a $C^1$ vector field not necessarily taking values in the span of $\{X_{H_k}, k=1,2\dots n\}$ the solution to the resulting equation
\begin{equa}
\label{sde2}
dy_t^\epsilon &= \sum_{i=1}^n X_{H_k}(y_t^\epsilon)\circ dB_t^k+V(y_t^\epsilon)dt+\epsilon K(y_t^\epsilon)dt,\\
y_t^\epsilon &= y_0
\end{equa}
will not conserve the energies.  On the other hand letting $\epsilon\to 0$, the deviation from level sets of the energies will be small.  Consider the solution $y^\epsilon(t/\epsilon)$  scaled in time by $1/\epsilon$, which has generator given by
${1\over \epsilon} \Lo_0+ \LL_K$. Note that the motion splits into two parts with the fast component an elliptic diffusion on the invariant torus and the slow motion governed by the transversal part of the vector field $ K$. The evolution of $y^\epsilon(t/\epsilon)$ is the skew product of the diffusion of order $1$ across the level sets and the fast elliptic diffusion of order $\epsilon^{-1}$ along the level sets. The motion on the level sets (thinking of the level sets as the standard n-torus),  which would be quasi-periodic if there were no diffusion terms,  is ergodic. The evolution of the action component of $\,y^\epsilon(t/\epsilon)$ will not depend on the angular variable in the limit as $\epsilon \to 0$ and is shown to be described by a system of $n$ ordinary differential equations whose right hand sides can be deduced from $\omega(K, X_{H_i})$, $i=1,\dots n$. Here $\omega$ is the symplectic 2-form. The convergence rate is shown to be of order $\epsilon^{1\over 4}$.

Furthermore if the vector field $K$ is given by a Hamiltonian function, the average of $\omega(K, X_{H_i})$ over the torus vanishes and we look at the second order scaling to see an interesting limit. The action component of $y^\epsilon(t\epsilon^{-2})$ involves a martingale term in the limit and
 the asymptotic law of $y^\epsilon(t\epsilon^{-2})$ across the level sets is shown to be given by a stochastic differential equation. It remains open to find an estimate for the rate of the convergence of the law of $y^\epsilon(t\epsilon^{-2})$ to the law of the limiting diffusion.

\medskip

{ \it Main Results. }
 Suppose that $\omega(V, X_{H_i})=0$ and $V$ commutes with each vector field $X_{H_i}$.  Let $y_0\in M$ be a regular point of $H$ with a neighbourhood $U_0$ the domain of an action-angle coordinate map.
 Let $T^\epsilon$ be the first time that the solution $y_{t\over\epsilon}$ starting from $y_0$ exits  $U_0$. 
  Set $H^\epsilon(t)=( H_1(y_{t/\epsilon}^\epsilon), \dots, H_n(y_{t/\epsilon}^\epsilon)).$  Then $H^\epsilon$ converges to the solution
of the following system of deterministic equations
$${d\over dt} \bar H_i(t)=\int_{M_{\bar H(t)} } \omega(X_{H_i}, K) (\bar H(t), z) \; d\mu_{\bar H_t}(z),$$
with corresponding initial condition:
If  $T^0$ is the first time that $\bar H(t)$ exits from $U_0$ then for all $t<T_0$,  $\beta>1$,
there exists a constant $C_2>0$ such that 
$$\left(\E(\sup_{s\le t} \|H^\epsilon(s\wedge T^\epsilon)-\bar H(s\wedge T^\epsilon)\|^\beta)\right)^{1\over \beta}\le 
C_2\epsilon^{1/4}.$$
Furthermore if  $r>0$ is such that  $U\equiv \{x: \|H(x)-H(y_0)\|\le r\}\subset U_0$  define
 $$T_\delta=\inf_t \{|\bar H_t- H(y_0)|\ge r-\delta\}.$$
 Then  for any $\beta>1$,   $\delta>0$ and a constant $C$ depending on $T_\delta$,
 $$P\left(T^\epsilon<T_\delta\right)\le C(T_\delta)\delta^{-\beta} \epsilon^{\beta/4}.$$

In the case of $K$ being a smooth local Hamiltonian vector field  \\  $\int_{M_{\bar H(t)} } \omega(X_{H_i}, K) (\bar H(t), z) \; d\mu_{\bar H_t}(z)=0$ and hence we could look at the second scaling. The law of the stochastic process $H(y^\epsilon_{t\over \epsilon^2})$ stopped at $S^\epsilon$, the first time that the process $y^\epsilon_{t\over \epsilon^2}$ exits from $U_0$, converges to that of $H(z_{t\wedge S^\epsilon})$ where $z_t$ is the solution to the following stochastic differential equation
$$dz_t^j=\sum_i \sigma^j_i(z_t)\circ dB_t^i+b_j(z_t)dt.$$
Here $(\sigma_i^j)$ is the square root of the matrices $(a_{ij})$,
\begin{eqnarray*}
a_{ij}(a)&=& -\int_{M_a} \omega(K,X_{H_j})\Lo_0^{-1}\Big(\omega(K, X_{H_i})\Big)(a,z) \; d\mu_a(z),
\end{eqnarray*}
and
$$b_j(a)={1\over 2}\int_{M_a} \LL_K\Lo_0^{-1}( \omega(X_{H_j}, K)) (a, z)  \, d\mu_a(z).$$

We give here a  somewhat trivial example of a stochastic integrable system of equations on $\R^{4}$ with the standard symplectic structure:

\begin{eqnarray*}
dx_1(t)&=&x_3(t)d B_t\\
dx_2(t)&=&x_4(t)dB_t+x_4 dW_t\\
dx_3(t)&=&-x_1(t)dB_t\\
dx_4(t)&=&-x_2(t)dB_t-x_2(t)dW_t
\end{eqnarray*}
where $B_t$ and $W_t$ are independent 1-dimenaional Brownian motions.
Take the perturbation vector to be   $ K_1=(0 ,{x_2\over x_2^2+x_4^2},0,0) $,  $K_2=({x_3\over (x_1^2+x_3^2)^2} ,0,{x_1\over (x_1^2+x_3^2)^2},0)$ or \\$K_3=({x_3^2\over (x_1^2+x_3^2)^{3\over 2}}, 0, -{x_1 \over (x_1^2+x_3^2)^{3\over 2}}, 0)$. In the case of $K_1$ we have a non-trivial average; in the case of $K_2$ and $K_3$  we have  trivial averages at first scaling and can proceed to the second scaling and obtain a SDE in the limit.
\medskip

{\it Remark On The Model .}  
This work is in the framework of Arnold on averaging principle of integrable Hamiltonian system as a stochastic Hamiltonian system can be considered as a family of ordinary differential equations with time dependent random vector fields (whose corresponding Hamiltonians are in general neither bounded from below nor differentiable in time).  Averaging of stochastic systems was pioneered by Khasminskii \cite{Khasminskii}, Papanicolaou, Stroock and Varadhan \cite{Papanicolaou-Stroock-Varadhan}, 
and has  been an active research field  on which there is a rich literature. The structure of the main averaging results are close to that described in the excellent survey of Papanicolaou \cite{Papanicolaou}.  
We refer to Abraham-Marsden \cite{Abraham-Marsden78}, Arnold \cite{Arnold-mechanics}, Hofer-Zehnder \cite{Hofer-Zehnder} and McDuff-Salamon \cite{McDuff-Salamon} as references for Hamiltonian systems on sympletic manifolds, 
to Givon-Kupferman-Stuart \cite{Givon-Kupferman-stuart} for some physical models behind these problems and for recent progress in the direction of deterministic averaging, to Freidlin-Wentzell \cite{Freidlin-Wentzell93} and  Sowers-Namachchivaya \cite{Srinamachchivaya-Sowers}  for random perturbations to  systems with one degree of freedom,   and to  Eizenberg-Freidlin \cite{Eizenberg-Freidlin93},   Borodin-Freidlin \cite{Borodin-Freidlin},
Freidlin-Wentzell\cite{Freidlin-Wentzell-book}, Sowers  \cite{Sowers02},  Koralov \cite{Koralov04}, Khasminskii-Krylov \cite{Khasminskii-Krylov}, and Khasminskii-Yin \cite{Khasminskii-Yin}  for recent related work on perturbations of stochastic systems as well as Arnold-Imkeller-Namachchivaya \cite{Arnold-Imkeller-Namachchivaya04} for a discussion on asymptotic expansion of a damped oscillator of one degree of freedom with small noise perturbation. For the Lagrangian mechanics and variational principle in stochastic framework we would like to refer to Bismut's work \cite{Bismut-mechanics}. However in this article we do not investigate  the stochastic mechanics related to the SDEs.

The main novelty of this work is the model itself. We also obtained a rate of convergence.  The generator of our perturbed system is:
$$\Lo_0^\epsilon=\sum_{k=1}^n {1\over 2} \LL_{X_{H_k}} \LL_{X_{H_k}}+\LL_V+\epsilon L_K,$$ 
from which we observe the following  aspects of the model:  a) 
The unperturbed random dynamical system is a completely integrable system.  b)  The fast component of the system is the diffusion term, not the deterministic term.  It is also worth noting that the limiting slow motion scaled at $1/\epsilon$  has $n$ degrees of freedom and is described by a system of $n$ ordinary differential equations.  It is only at the second scaling, in the case of the perturbation being Hamiltonian, that we see a limiting $n$-dimensional non-trivial Markov process. The convergence of the slow motion $(H_1(y^\epsilon_{t\over \epsilon}), \dots, H_n(y^\epsilon_{t\over \epsilon}))$ is in $L^p$ with rate $\epsilon^{1\over 3}$. At the second scaling the slow motion converges weakly.  The assumptions we make on the Hamiltonian functions are:   the $\R^n$ valued function $(H_1,\dots, H_n)$ is proper and its set of critical points has measure zero. 

 The following work relates particularly well to ours.
  We'll point out the differences and similarities.
 In Dolgopyat \cite{Dolgopyat} the following is proved:  If $\hat h$ is a first integral of   $\dot y=\E (F(y, \xi_1))$, where $\xi_1$ is a random variable of compact support, the piece-wise linear function $h_\epsilon$ given by $\hat h_\epsilon(\epsilon^2 n)=\hat h(x_n^\epsilon)$ converges weakly to the solution of a SDE, under suitable conditions. This was proved using the martingale method.  Here $x_n^\epsilon$ is solution to
 $$x_{n+1}=x_n+\epsilon F(x_n, \xi_n)+\epsilon^2 G(x_n, \xi_n)+\epsilon^3 H(x_n,\xi_n, \epsilon),$$
and the $\xi_n$'s are i.i.d. random variables.    
The equation governing $x_n$ are very general.  If the perturbation $K$ in (\ref{sde2}) is given by a Hamiltonian function we may take $F\equiv 0$ in Dolgopyat's model and take $\hat h=h$, a first integral to (\ref{1});  however
 it is not clear how the piece-wise linear function $\hat h_\epsilon$ relates to the fast motion $h(x_{t/\epsilon^2})$.
In Eizenberg-Freidlin \cite{Eizenberg-Freidlin93} and Borodin-Freidlin \cite{Borodin-Freidlin},  the diffusion part of the motion belongs to the slow component and the fast motion is deterministic. More precisely the perturbed  generator is $L_V+\epsilon L_1+L_2$ for a vector field $V$ and an interaction term $L_2$.  Freidlin and Weber \cite{Freidlin-Weber01} \cite{Freidlin-Weber04} have a different objective. They are mainly concerned with one conserved quantity $H$, not a completely integrable system. The objective there is to obtain a limiting Markov process, using weak convergence, on the graph homeomorphic to the set of connected components of the level sets of $H$. 

\section{Preliminaries}
\subsection{Hamiltonian and Symplectic Vector fields} 
 Every symplectic manifold has a natural measure, called the Liouville measure. It is in fact $\,\wedge^n \omega$, differing from the volume form by a constant. Denote by $\iota _v\omega $  the inner
product of a tangent vector $v$ with $\omega $. The map from $TM$ to $T^{*}M$ given by $v\mapsto \iota _v\omega $ is a vector bundle isomorphism, and there is a one to one correspondence between vector fields and differential 1-forms. A {\it symplectic vector field} $V$, also called a {\it locally Hamiltonian vector field}, is one which preserves the symplectic structure, \ie $L_V\omega=0$. Here $L_V$ denotes Lie differentiation in the direction of $V$. Equivalently $\iota_V\omega$ is a closed differential 1-form.  For every $C^1$ function $H:M\to {\Bbb R}$ we can associate a {\it Hamiltonian vector field} (also called  {\it symplectic gradient vector field}) given by: 
$$\iota _{X_H}\omega =dH.$$
The canonical sympletic structure on $\R^{2n}$ with coordinates $(q_1,\dots, q_n, p_1,\dots, p_n)$ is $\omega=\sum_{i=1}^n dq_i\wedge dp_i$. Darboux's theorem asserts that any symplectic manifold is locally $\R^{2n}$ with its canonical sympletic structure. If the first de Rham cohomology ${\mathbb H}^1(M;\R)$ vanishes, as in the case of $\R^{2n}$, every locally Hamiltonian vector field is given by a Hamiltonian function. There are locally Hamiltonian vector fields which are not given by a Hamiltonian function. Take the two torus $T^2$ with coordinates $x$ and $y$ periodic in $x$ and $y$. The canonical sympletic structure on $\R^2$ induces the symplectic structure on $T^2$: $\omega=dx\wedge dy$.
A vector field $X(x,y)=a(x,y){\partial \over \partial x}+b(x,y){\partial \over \partial y}$ with ${\partial a\over \partial x} +{\partial b\over \partial y}\equiv0$ is clearly locally Hamiltonian: the 1-form $\iota_X\omega=a(x,y)dy-b(x,y) dx$ is closed. If $a=1, b=0$, the vector field is not given by a Hamiltonian function on $T^2$.

The space of smooth functions on $M$ has a Lie algebra structure given by the Poisson bracket. 
The Poisson bracket of two smooth functions is denoted by 
 $\{F_1, F_2\}$ and  
 $\{F_1,F_2\}=dF_1(X_{F_2})=\omega(X_{F_1}, X_{F_2})$.
The vector field corresponding to the Poisson bracket is precisely the Lie bracket of the Hamiltonian vector fields $X_{F_1}$ and $X_{F_2}$. Two Hamiltonian functions are {\it Poisson commuting} or {\it in involution} if their Poisson bracket vanishes, in which case their corresponding Hamiltonian flows commute. If $\{F, H\}=0$ we say that $F$ is a {\it first integral} of $H$.  Two Hamiltonian functions are said to be linearly independent at $x$ if their associated Hamiltonian vector fields are linearly independent at that point. A family of $n$ Hamiltonian functions is said to form an {\it  integrable system} if the Hamiltonian functions are pairwise Poisson commuting and if they are linearly independent on a set of full measure.

\subsection{An example of a stochastic Hamiltonian system on $\R^{2n}$}
\label{se:example}
The Hamiltonian vector field  given by an Hamiltonian function $H$ is given by $X_H=JdH$ where $J$ is the canonical complex structure: 
\[J=
\left(
\begin{array}{rl}
 0 &  {\mathbf 1}    \\
- {\mathbf 1} & 0,     
\end{array}
\right)
\]
where ${\mathbf 1}$ denotes the $n$ by $n$ identity matrix. The corresponding Hamiltonian system thus takes the familiar form
\begin{eqnarray*}
\dot q_i&=&{\partial H\over \partial p_i}, \qquad 1\le i\le n\\
\dot p_i&=&-{\partial H\over \partial q_i}, \qquad 1\le i
\le n.
\end{eqnarray*} 
For simplicity write $p=(p_1,\dots, p_n)$ and $q=(q_1,\dots, q_n)$. An important class of examples of Hamiltonian functions on $\R^{2n}$ is those of the form $H(p,q)={1\over 2}|p|^2+V(q)$ for some potential function $V$. If $V$ is quadratic, \eg $V(q)={1\over 2} a^2|q|^2$, we have the standard harmonic oscillator. The Poisson bracket in $\R^{2n}$ is of the following form:
$$\{H,F\}=\sum_{i=1}^n\left( {\partial H\over \partial p_i}{\partial F\over \partial q_i}-{\partial H\over \partial q_i}{\partial F\over \partial p_i}\right).$$ 
A example of an integrable stochastic Hamiltonian system is given by
\begin{eqnarray*}
d q_i(t)&=&{\partial K\over \partial p_i}dt+\sum_{k=1}^n{\partial H_k\over \partial p_i}\circ dB_t^k\\
d p_i(t)&=&-{\partial K\over \partial q_i}dt+\sum_{k=1}^n {\partial H_k\over \partial q_i}\circ dB_t^{k}.
\end{eqnarray*}
where  \begin{eqnarray*}
H_1&=&{1 \over 2}\sum_{i=1}^n a_i^2q_i^2+{1\over 2}\sum_{i=1}^n p_i^2\\
H_k&=&{1 \over 2} a_kq_k^2+{1\over 2} {p_k^2\over  a_k}, \qquad 2\le k<n,
\end{eqnarray*}
and $K$  is a smooth function which commutes with all $H_i$'s, \eg if $K$ is a smooth function of $H_i$'s.

\subsection{The invariant manifolds and integrable symplectic Hamiltonian systems} 
Let $\{H_k\}_{k=1}^n$ be an integrable family of smooth Hamiltonian functions, \ie they are Poisson commuting and so the $H_k$'s  are first integrals of each other and they are independent on a set of full measure. 
 For $a=(a_1,\dots, a_n)\in \R^n$ denote by $M_a$ the level set of the first integrals $\{H_k\}$:
$$M_a=\cap_{i=1}^n\{x: H_i(x)=a_i\}.$$
The Liouville-Arnold theorem states that if $\{H_k\}_{k=1}^n $ are independent on $M_a$ then $M_a$ is a smooth manifold and furthermore it is diffeomorphic to an $n$ dimensional torus if it is compact and connected. For such value $a$, $M_a$ is invariant under the flows of each $H_k$ and each $x$ in $M$ determines an invariant manifold through the value 
$a=(H_1(x), \dots, H_n(x))$, which we write also as $M_{H (x)}$.  

An application of It\^o's formula below shows that the solution flow $\{F_t(x): t\ge 0\}$ of (\ref{1}) preserves the invariant manifolds $\{M_a\}$:
\begin{eqnarray*}
dH_i(x_t)& = &\sum_kdH_i\Big( X_{H_k}(x_t)\Big)\circ dB_t^k+dH_i\big(V(x_t)\big)dt=0, \qquad 1\le i\le n.
\end{eqnarray*}

 For simplicity we  assume throughout the paper the following:
\begin{itemize}
\item   The invariant manifolds are compact,
\end{itemize}
 which is the case if the map $x\in M\mapsto \big(H_1(x),\dots H_n(x)\big)\in \R^n$ is proper.
 Note that  the $n$ vector fields $\{X_{H_i(x)}\}$ are tangent to $M_{H(x)}$ and the symplectic form $\omega$ vanishes on the invariant manifolds $M_a$.
Therefore the stochastic differential equation (\ref{1}) is elliptic when restricted to individual invariant manifolds and the Markovian solution is ergodic. Denote by $\mu_a$ the unique invariant probability measure on $M_a$; it can be considered as the uniform probability measure on the torus.

\subsection{The invariant measure and th divergence operator for semi-elliptic stochastic symplectic systems }
\label{se:measure}
Let $\{A^0, A^1, \dots, A^n\}$ be smooth symplectic vector fields with $[A^i, A^j]=0$ for all $i, j$. Assume that $\{A^1,\dots ,A^n\}$ spans a sub-bundle $E$ of the tangent bundle $TM$ of rank $n$.  Consider the following stochastic differential equation:
\begin{equation}\label{eq:symp}
dx_t=\sum_{i=1}^n A^i(x_t)\circ dB_t^i +A^0(x_t)dt.
\end{equation}
If there is a global solution flow $\{F_t(x_0, \omega): t\ge 0\}$  to equation (\ref{eq:symp}), then the solution flows are stochastic symplectomorphisms, \ie $\omega=F_t^*\omega$, where $\omega$ is the symplectic form, not the chance variable.

For each $x\in M$, define a linear map $A(x): \R^{2n}\to T_xM$ by
$$A(x)(e)=\sum_{i=1}^n A^i(x)\langle e, e^i\rangle, \qquad e\in \R^{2n}$$
where $\{e^i\}$ is an orthonormal basis of $\R^{2n}$. The linear map is onto $E_x$ with kernel $\{0\}\times \R^n$ and gives a positive symmetric bilinear form on $E$ by making $\{A^i(x)\}$ an orthonormal basis:
$$\langle A^i(x), A^j(x)\rangle=\delta_{ij}.$$
Then $A(x)$ is an isomorphism from $\R^n\times \{0\}$ to $E_x$. This defines a metric on $E$: for $u= \sum_iu_i A^i$ and $v=\sum_i v_i A^i$, $$\langle u, v \rangle=\sum_{i=1}^n u_iv_i$$
and for a function $f$ we define its gradient $\nabla_E f=\sum_i df(A^i)A^i$.
The symplectic structure $\omega$ restricts to $E$  defines a complex structure on $M$ as following: we first give the tangent bundle $TM$ any Riemannian metric which agrees with the one constructed on $E$ using the linear map $A$. Define
$J_x: T_xM \to T_xM$ by
$$\omega(J_xu,v)=\langle u, v\rangle_x.$$
To see that this identity defines $J_x$ uniquely suppose that 
for $u\in T_xM$ there are $u_1$ and $u_2$ satisfies
 $\omega(u_i,v)=\langle u_i, v\rangle_x$, $i=1,2$. Then $\omega(u_1-u_2, v)=0$ for all $v$. Thus $u_1=u_2$. Existence can be easily seen as direct calculations can be done in $\R^{2n}$.
 
 Next take $A^i=X_{H_i}$ in (\ref{eq:symp}), $i=1,2,\dots n$, to be the Hamiltonian vector 
 fields for an integrable family of Hamiltonian functions $\{H_i\}$ and $A^0=V$. We arrive back 
 to the integrable stochastic sympletic equation (\ref{1}) where $V$ is a symplectic vector 
 field commuting with all $X_{H_i}$'s. 
Under our assumption that $$H: x\to \big(H_1(x), \dots H_n(x)\big)$$ is a proper 
map,  then for almost every point $a_0$ in $\R^n$ it  is either trvial or a local trivial 
fibration in the sense that there is a neighbourhood $V$ of $a_0$ such that 
$H^{-1}(a)$ is a smooth sub-manifold for all $a\in V$ and that there is a 
diffeomorphism from $ H^{-1}(V)$ to $V\times  H^{-1}(a_0)$. Such $a_0$  is 
called a regular value of $H$. Denote by $\Sigma_{H}$ the set of values in 
$\R^n$ which are not regular. A point $y$ in $M$ is said to be a critical point if
$H(y)\in \Sigma_{H}$.  By Sard's theorem the set of critical values of 
the function $ H$ has measure zero. The $2n$-differential form $\omega^n$, as a 
measure, has a decomposition which gives a measure on each invariant manifold $M_a$ for regular 
$a$ value. The decomposition can be chosen in the following way. First recall that on a 
neighbourhood of a regular point $a_0$ of $H$, every component of the level set 
$M_{a_0}$ is diffeomorphic to an $n$-torus and a small neighbourhood $U_{0}$ of $M_{a_0}$ is 
diffeomorphic to the product space $T^n\times D$ where $D$ is a relatively compact open set in 
$\R^n$, see \eg \cite{Arnold-Kozlov-Neishtadt}. More precisely 
if $V$ is an open set of $\R^n$ such that $H^{-1}(V)$ does not contain any critical 
points of $H$ then it is diffeomorphic to $D\times T^n$.
Take an action angle chart around $M_a$ which is diffeomorphic to $D\times T^n$ for some open 
set $D$. The measure $(\sum_i dI^i\wedge d\theta^i)^n$
on the product space naturally splits to give us a probability measure,  the Haar measure 
$d\theta^1\wedge\dots d\theta^n$ on $T^n$. We take the corresponding one on $M_a$ and denote
it by $\mu_a$.
Let $U$ be a section of $E$. Define the divergences $\div_E^a U$  to be the functions such that
$$\int_{M_a} df(U)\;d\mu_a=-\int_{M_a} f{\div}_E^a U \;d\mu_a$$
for all  smooth functions $f$ on $M_a$.
Note that $\div_E X_{H_i}=0$, since 
$$\int_{M_a} df(X_{H_i}) d\mu_a=\int_{M_a} \{H_i, f\}d\mu_a=0$$ for all smooth functions $f$ (see the beginning of section \ref{se:perturbation} for a calculation). Thus if $U=\sum_i a_iX_{H_i}$
where $a_i$ are constant on $M_a$ and is thus divergence free.

\begin{remark}
 \label{le-0}
 Let $U$ be a  section of $E$ commuting with all $X_{H_i}$  the invariant measure for the SDE (\ref{1}) restricted to the invariant manifold $M_a$ is $\mu_a$, which varies smoothly with $a$ in sufficiently small neighbourhoods of  a regular value.
\end{remark}

\begin{proof}
The measure $\omega^n$ is an invariant measure for the SDE on $M$ due to the fact that the solution of the SDE leaves invariant the symplectic form. More precisely, since $U$ commutes with $\{X_{H_i}\}$ and thus can be written in the the form of $U=\sum_i a_iX_{H_i}(x)$, where $a_i$ are constant on $M_a$, it is therefore divergence free. Thus the invariant measure of the SDE restricted to the torus is the same as that of the corresponding SDE without a drift. From the action angle transformation we see that the measure $\mu_a$ is an invariant measure for the SDE restricted to $M_a$. This is in fact the only invariant measure for the SDE on $M_a$ since the system is elliptic when restricted to each level set and the conclusion follows.
\end{proof}

\section{An Averaging Principle}
\label{se:average1}
Let $\{H_i\}_{i=1}^n$ be a completely integrable system on a smooth $2n$-dimensional symplectic manifold $M$ so that the functions $\{H_i\}$ are constants of motions of each other and that they are pairwise in involution.    We assume that the $\R^n$-valued function $H=(H_1,\dots, H_n)$ is proper  and its set of critical points has measure zero. Note that the vector fields $\{X_{H_i}\}$ form an integrable distribution and  through each point of the manifold there is an integrable $n$ dimensional manifold. 
 

Take an action-angle coordinate: $\phi^{-1}: U_0\to D\times T^n$. In this coordinate,
 $x=\phi(I, \theta)$, $I\in D$, $\theta\in T^n$, and
 $(\phi^{-1})_*\omega=dI\wedge d \theta$ defines a symplectic structure on $ D\times T^n$. Furthermore if $\tilde H_i=H_i\big(\phi(I, \theta)\big)$ is the induced Hamiltonian on $D \times T^n$ then
 $\dot I_i^k=-{\partial \tilde H_i \over \partial \theta_k}=0$ and
 \begin{equation}
 \label{omega}
\dot \theta^k_i ={\partial \tilde H_i \over \partial I_k}=\omega^k_i(I)
\end{equation}
 with $\omega_i^k$ smooth functions. In fact $X_{\tilde H_i}=(\phi^{-1})_*(X_{H_i})=-\sum_{k=1}^n  {\partial (H_i\circ \phi)\over \partial I_k}{\partial \over \partial \theta_k}$.
For example the integrable Hamiltonian system in section \ref{se:example} is equivalent to the Hamiltonian system 
$H_1=\sum_{i=1}^n a_i \bar q_i$, $a_i>0$, and $H_k= \bar q_k, k=2,\dots, n$,
through the action angle coordinates change $(q,p)\mapsto (\bar q, \bar p)$:
\begin{eqnarray*}
&&\left( q_1, \dots  q_n,\ p_1,\dots,  p_n\right)\\
& = & \left(\sqrt{2  \bar q_1 \over a_1}\cos \bar p_1,\dots
\sqrt{2 \bar q_n \over a_n}\cos \bar p_n,\sqrt{2a_1 \bar q_1}\sin  \bar p_1, \dots,
\sqrt{2a_n \bar q_n}\sin \bar p_n,\right).
\end{eqnarray*}
The corresponding Hamiltonian system is the trivial one $\dot {\bar p_i}=a_i, \dot {\bar q_i}=0$.
Since $U_0$ is diffeomorphic to $D \times T^n$ there is a constant $r>0$ such that $U_0$ contains
the open set $\{x: \sum_i|H_i(x)-H_i(y_0)|^2\le r^2\}$.

 Let $K$ be a smooth vector field, transversal in the sense that 
 $\omega( X_{H_i}, K)$ are not all identically zero. Denote by $y_t^\epsilon$ the solution to (\ref{sde2}), the perturbation of the integrable system (\ref{1}) starting from a given point $y_0$ in $M$. Set
 $x_t=y_t^0$, the solution to (\ref{1}) with initial value $y_0$. If $V$ is a vector field on $M$ denote by $\tilde V$  the induced vector field on $D\times T^n$. We assume the following of the SDE (\ref{sde2}):\\
 
 \noindent 
{\bf Condition R:} Suppose that $\omega(V, X_{H_i})=0$ and $V$ commutes with all vector fields $X_{H_i}$.  Let $y_0\in M$ be a regular point of $H$ with a neighbourhood $U_0$ the domain of an action-angle coordinate map: $\phi^{-1}: U_0\to D\times T^n$, where $D$ is an open set of $\R^n$.

We adopt the following notation: if $f$ is a function on $U_0$, by $\tilde f$ we mean the representation of $f$ in $D\times T^n$.

\begin{lemma}
\label{lemma-1} Assume condition $R$ holds for (\ref{sde2}).
Let $\tau^\epsilon$ be the first time that the solution $y_t^\epsilon$ starting from $y_0$ exits  $U_0$.    Then for any smooth function $f$  on $M$, 
\begin{itemize}
\item[(1)] 
$$\left[\E\left(\sup_{s\le t\wedge \tau^\epsilon}|f(y_s^\epsilon)-f(x_s)|^p\right)\right]^{1\over p}
\le C_1 \epsilon (t+ t^2),$$
where $C_1=C_1(V, K, H_i,f)$ depends on the upper bounds of the functions 
$|d \tilde f|$, $|{\partial^2 \tilde H_k\over \partial I_i\partial I_j}|$, $|d \tilde V|$, $| \tilde K|$ on $D\times T^n$.
\item[(2)] If $V\equiv 0$, then the estimates above, $\epsilon (t+ t^2)$, can be improved to   $C_1 \epsilon (t+ t^{3\over 2})$.
\end{itemize}
\end{lemma}

\begin{proof}
In the proof below $C$ stands for an unspecified constant. 
We denote the flows in action-angle coordinates by $x_t=\phi(I_t, \theta_t)$ and
$y_t^\epsilon=\phi(I_t^\epsilon, \theta_t^\epsilon)$.  Set $\tilde f=f\circ \phi$. Then
\begin{eqnarray*}
|f(y_t^\epsilon)-f(x_t)|&=&|\tilde f(I(y_t^\epsilon), \theta(y_t^\epsilon))-\tilde f(I(x_t), \theta(x_t))|\\
&\le& C | I(y_t^\epsilon)-I(x_t)|+C|\theta(y_t^\epsilon)-\theta(x_t)|,
\end{eqnarray*}
using the fact that ${\partial \tilde f\over \partial I}$  and ${\partial \tilde f\over \partial \theta}$ are bounded on $T^n\times D$ as $D$ is relatively compact.  In the local chart, ${\partial \tilde V\over \partial \theta_i}=0$  and we can write $V(I,\theta)=V_j(I){\partial H_j\over \partial I}(I,\theta)=\omega^j_0(I){\partial \over \partial \theta_j}$ for some smooth functions $\omega^j_0$ on $D$. 
The perturbation vector field  can be written as $(K^\theta, K^I)$ where $K_\theta=(K_\theta^1, \dots, K_\theta^n)$ and $K_I=(K_I^1,\dots, K_I^n)$ be respectively the angle and the action component of the vector field $\tilde K$ on $T^n\times D^n$.
The result is now clear from the form of the SDE on $T^n\times D$:
\begin{eqnarray*}
d I_t^{\epsilon,i}&=&\epsilon \; K_I^i(I_t^\epsilon, \theta_t^\epsilon)\, dt,\\
d\theta_t^{\epsilon,i}&=&\sum_{k=1}^n\omega^i_k(I_t^\epsilon)\circ dB_t^k+\omega_0^i(I_t^\epsilon)\,dt+\epsilon K_\theta^i(I_t^\epsilon, \theta_t^\epsilon)\,dt,
\end{eqnarray*}
where $\omega^i_k, i,k=1,\dots n$, are defined by (\ref{omega}).
Indeed, then 
$$\sup_{s\le t\wedge \tau^\epsilon}|I_s^\epsilon-I_s|
=\epsilon \sup_{s\le t\wedge \tau^\epsilon}|\int_0^s \big| K_I(I_s^\epsilon, \theta_I^\theta)\big|ds \le \epsilon t \sup_{D\times T^n}  |K_I| ,$$
and  for $s<\tau^\epsilon$,
 \begin{eqnarray*}
\theta_s^{\epsilon,i}-\theta_s^i
&=&\sum_{k=1}^n \int_0^s  \Big( \omega^i_k(I_r^\epsilon)-\omega_i^k(I_r) \Big)\circ dB_r^k\\
&&+\int_0^s \Big(\omega_0^i(I_r^\epsilon)-\omega_0^i(I_r)\Big)dr
+\epsilon \int_0^s  K_\theta^i(I_r^\epsilon, \theta^\epsilon_r)\; dr.
\end{eqnarray*}

 As  the Stratonovitch correction for the SDE vanishes, we may replace the Stratonovitch integration by It\^o integration:
  $$\int_0^s  \Big( \omega^i_k(I_r^\epsilon)-\omega^i_k(I_r) \Big)\circ dB_r^k=\int_0^s  \Big( \omega^i_k(I_r^\epsilon)-\omega^i_k(I_r) \Big) dB_r^k.$$ 
Consequently,
\begin{eqnarray*}
\Big|\theta_i(y_s^\epsilon)-\theta_i(x_s)\Big|
&\le& \Big|\sum_{k=1}^n \int_0^s  \Big( \omega^i_k(I_r^\epsilon)-\omega_i^k(I_r) \Big) dB_r^k\Big|\\
&&+\sup_{D\times T^n} |d\omega_0^i|\cdot \int_0^s \big |  I_r^\epsilon-I_r \big| \, dr
+\epsilon s \sup_{D\times T^n}\Big| K^i_\theta \Big|\\
&\le& \Big|\sum_{k=1}^n \int_0^s  \Big( \omega^i_k(I_r^\epsilon)-\omega_i^k(I_r) \Big) dB_r^k\Big|\\
&&+\epsilon {s^2\over 2} \sup_{D\times T^n}  |K_I|\cdot
\sup_{D\times T^n} |d\omega^0_i|+\epsilon s  \sup_{D\times T^n}\Big| K^i_\theta \Big|.
\end{eqnarray*}
Summing up over $i$, we have
 \begin{eqnarray*}
 &&\E\sup_{s\le t\wedge \tau^\epsilon}|\theta_s^\epsilon-\theta_s|^p\\
 &\le&C_1 \sup_{s\le t }
 \Big(\sum_{i,k=1}^n \E  \int_0^{s} \Big| \omega^i_k(I_r^\epsilon)-\omega^i_k(I_r)\Big|^2 \Big)^{p/2}
 +C_2(\tilde K)\epsilon(t+t^2)^p\\
&\le&C_1 \Big(\sum_{i,k}(|d\omega_k^i|\vee1)^p\Big) \epsilon^{p} t^{3p\over 2}+C_2(\tilde K) \epsilon^p(t+t^2)^p
\end{eqnarray*}
by $L_p$ inequalities for martingales. Combining the estimates we obtain $$\E\sup_{s\le t\wedge \tau^\epsilon}|f(y^\epsilon_{s})-f(x_s)|\le C_3\epsilon (t+t^2)$$
for some constant $C_3$.

(ii) If the drift $V\equiv 0$ then $\omega_0=0$ and the calculation above shows that the estimate is of the order $\epsilon (t^{3\over 2}\wedge 1)$.
\end{proof}

If the stochastic dynamical system (\ref{1}) is subjected to a small non-Hamiltonian perturbation, the slow variable is in the direction transversal to the energy surfaces while the stochastic components are the fast variables. The lemma shows that the first integrals of the perturbed system change by an order $ \epsilon(t+t^2)$ over a time interval $t$ and so the slow component accumulates over a time interval of the size $ t/\epsilon$ and we obtain a new dynamical system  in the limit: as $\epsilon$ goes to zero the motion along the torus is significantly faster compared to the motion in the transversal direction and thus the action component of $y_{t/\epsilon}^\epsilon$ has a limit as the randomness in the fast component is averaged out by the induced invariant measure, as shown below.
Recall that $H(x)=(H_1(x),\dots, H_n(x))$. \\

We first prove a lemma:
\begin{lemma}\label{lemma-2}
 Assume condition $R$ holds. Let  $g$ be a real-valued function on $M$, which is considered in the action angle co-ordinates as a function from $D\times T^n$ to $\R$.  Define $Q^g: D\subset \R^n \to \R$ to be its average over the torus: 
\begin{equation}\label{averaged-1}
Q^g(a)=\int_{T^n}   \tilde g(a,z) \; d\mu(z).
\end{equation} 
Suppose that $g$ is $C^1$ on $U_0$. Set $$ H^\epsilon_i(s)=H_i(y_{s/\epsilon}^\epsilon), \qquad 
 H^\epsilon(s)=( H^\epsilon_1(s), \dots, H^\epsilon_n(s)).$$ Then
\begin{equation}
\label{estimate-2}
\int_{s\wedge T^\epsilon}^{(s+t)\wedge T^\epsilon} g(y_{r/\epsilon}^\epsilon) dr
=\int_{s\wedge \tau^\epsilon}^{(s+t)\wedge \tau^\epsilon}  Q^g\big (H^\epsilon(r)\big) \;dr+\delta(g,\epsilon,t)
\end{equation}
 with the following rate of convergence: for any $\beta>1$,
\begin{equation}\label{error-2}
\Big(\E \sup_{s \le t } \big |\delta(g,\epsilon,s)\Big|^\beta\Big)^{1\over \beta}
\le C(t)\epsilon^{1\over 4}.
\end{equation}
where $T^\epsilon$ is the first time that $y^\epsilon_{t/\epsilon}$ exits from $U_0$ and $\tau^\epsilon=T^\epsilon /\epsilon$.
\end{lemma}

\begin{proof} 
The idea is to approximate $g(y^\epsilon_r)$ by $g(y_r)$ on sufficiently small intervals and to apply the law of large numbers to each integral bearing in mind that the cost to replace   $y^\epsilon_r$ by $y_r$ is of order $\epsilon \delta $ on an interval of size $\delta$, assuming that $\delta\to \infty $  as $\epsilon \to 0$, and that the error term for replacing time average by space average is of order   ${1\over \sqrt \delta }$.

Let $\tau^\epsilon$ be the first time that $y^\epsilon_t$ exits from $U_0$. For $q\in (0, 1)$,
let $\Delta t={{(t+s)\over \epsilon^q}\wedge \tau^\epsilon-{s\over \epsilon^q}\wedge \tau^\epsilon}$, which is of order $\epsilon^{-q}$,   and set $N\equiv N(\epsilon)=[ \epsilon^{q-1}]+1$ which is of order $\epsilon^ {q-1}$. Here  $[\epsilon^{q-1}]$ is the integer part of $ \epsilon^{q-1}$ and all terms may depend on the sample paths of $\omega$. Take $t_n={{s\over \epsilon}\wedge \tau^\epsilon+n \Delta t}$, $1\le n\le N-1$, so that  
$${s\over \epsilon} \wedge \tau^\epsilon =t_0<t_1<\dots <t_{N-1}
<{{s+t\over \epsilon}\wedge \tau^\epsilon}.$$
 We first make some pathwise estimates. For any $C^1$ function $g$ on $M$,
\begin{eqnarray*}
\left|\int_{s\wedge T^\epsilon }^{(s+t)\wedge T^\epsilon } g(y^\epsilon_{u/\epsilon}) du\right|
&=&\left |\epsilon\int^{{s+t \over \epsilon}\wedge \tau^\epsilon}_{{s\over \epsilon}\wedge \tau^\epsilon} g(y_r^\epsilon)dr\right|\\
&\le &\epsilon \left| \sum_{n=0}^{N-1}\int_{t_n}^{t_{n+1}} g(y_r^\epsilon)\;dr\right |+\epsilon \left |\int_{t_N}^{{s+t\over \epsilon}\wedge \tau^\epsilon }  \right|
g(y_r^\epsilon)\;dr
\end{eqnarray*}

Since  $g$ is bounded on $U_0$, the second term on the right hand side of  the above expression converges to zero with rate $\epsilon^{1-q}$:
$$\epsilon\Big|\int_{t_N}^{{(s+t)\over \epsilon} \wedge \tau^\epsilon}  g(y_r^\epsilon)\;dr\Big| \le C \epsilon \Delta t \le Ct\epsilon^{1-q}.$$
For the remaining terms we use the splitting
 \begin{eqnarray*}
\epsilon \sum_{n=0}^{N-1} \int_{t_n}^{t_{n+1}}g\left (y_{r}^\epsilon\right) \; dr
&=&\epsilon \sum_{n=0}^{N-1}\int_{t_n}^{t_{n+1}}\left[ g(y_{r}^\epsilon)- g\left (F_{r-t_n}(y_{t_n}^\epsilon, \Theta_{t_n}(\omega)) \right) \right]\;dr\\
&&+\epsilon \sum_{n=0}^{N-1}\int_{t_n}^{t_{n+1}} g\left (F_{r-t_n}(y_{t_n}^\epsilon, \Theta_{t_n}(\omega)) \right) \; dr.
\end{eqnarray*}
where $\omega$ denotes the chance variable,
 $\Theta_t$  the shift operator on the canonical probability space: $\Theta_t(\omega)(-)=\omega(-+t)-\omega(t)$, and $\{F_t(x,\omega), t\ge 0 \}$  the solution flow of the unperturbed stochastic differential equation (\ref{1}) with starting point $x$. Write the summation as the sum of $A_1$ and $A_2$ and the second term is:
  \begin{eqnarray*}
A_2(t, \epsilon)&\equiv&\epsilon \sum_{n=0}^{N-1}\int_{t_n}^{t_{n+1}} g\left (F_{r-t_n}(y_{t_n}^\epsilon, \Theta_{t_n}(\omega)) \right) \; dr.
\end{eqnarray*}

The law of the large numbers says that for any function $f$ on a compact manifold converging to infinity as $t$ converges to infinity,
${1\over t }\int_{s}^{s+t}f(x_r)dr$ converges to 
$\int_{M} f(z)dz$ when $t\to \infty$ with  rate ${1\over \sqrt t}$ and the convergence is uniform on compact time intervals in $L^p$ for all $p>1$. Here $dz$ is the volume measure. The easiest way to see that this holds is to first assume that $\int f dz$ vanishes and let $h$ be the function solving $\Delta h=2f$ and apply It\^o's formula to $h(x_t)$ on
the time interval $[ s, s+t]$. Note that the intervals $[t_i, t_{i+1}]$ are either constant intervals or intervals of zero length with the exception of one bad interval of the form $[a, \tau^\epsilon]$ of size at most $\Delta t$. Let $M$ be the integer such that $[t_i, t_{i+1}]$ are constant intervals for $i<M$.   For the bad interval $[t_{M}, \tau^\epsilon)$, 
 the following term in $A_2$ 
$$\epsilon 
 \int_{t_M}^{\tau^\epsilon} g\left (F_{r-t_n}(y_{t_n}^\epsilon, \Theta_{t_n}(\omega)) \right) \; dr $$
 is of order $\epsilon \Delta t$.
On each constant interval $[t_i, t_{i+1}]$, $i<M$, the corresponding term in $A_2$ is:
$$\epsilon \int_{t_n}^{t_{n+1}} g\left (F_{r-t_n}(y_{t_n}^\epsilon, \Theta_{t_n}(\omega)) \right) \; dr= \epsilon \int_0^{\Delta t} g\left (F_r(y_{t_n}^\epsilon, \Theta_{t_n}(\omega)) \right) \; dr.$$
Denote by   $\mu_{ H^\epsilon( \epsilon t_n)}$, or  $ \mu_{y_{t_n}^\epsilon}$, the invariant measure on the invariant manifold $M_{  H^\epsilon(\epsilon t_n)}\equiv M_{y_{t_n}^\epsilon}$.
We apply the law of large numbers to such terms and use  the Markov property of the flow to obtain the following estimates, for all sufficiently small $\epsilon$, 
\begin{eqnarray*}
&&\left[\E \sup_{u\le t} \Big(A_2(u, \epsilon)
-\epsilon \Delta t \sum_{n=0}^{N-1}\int_{M_{ H^\epsilon(\epsilon t_n)}  }g(H^\epsilon(\epsilon t_n), z) d\mu_{ H^\epsilon(\epsilon t_n)}(z)\Big)^\beta\right]^{1\over \beta}\\
&&\le C \epsilon \Delta t+\\
&& N \sup_n  \left[\E  \Big(\epsilon
 \int_0^{\Delta t} g\left (F_r(y_{t_n}^\epsilon, \Theta_{t_n}(\omega)) \right) \; dr-\epsilon \Delta t \int_{M_{ H^\epsilon(\epsilon t_n)}  }g(H^\epsilon(\epsilon t_n), z) d\mu_{ H^\epsilon(\epsilon t_n)}(z)\Big)^\beta\right]^{1\over \beta}\\
&&\le C\epsilon \Delta t+\\
&& (\epsilon   \Delta t) N
\sup_n \left( \E\sup_{u\le t}  \Big|{1\over \Delta t}
\int_0^{\Delta t} g\left (F_{r}(y_{t_n}^\epsilon, \Theta_{t_n}(\omega)) \right)  dr \;
- \,\int_{M_{ H^\epsilon(\epsilon t_n)}  }g( H^\epsilon(\epsilon t_n),z) d\mu_{ H^\epsilon(\epsilon t_n)}(z)\Big|^\beta\right)^{1\over \beta} \\
&&\le Ct\epsilon^{1-q}+ C {\epsilon \Delta t N\over  \sqrt{t /\epsilon^q}} 
\le Ct\epsilon^{1-q}+ C { \sqrt{t }} \epsilon^{q\over 2} 
\end{eqnarray*}
On the other hand letting $s_n=\epsilon t_n$ so $\Delta s=\epsilon \Delta t$ and consider $s_0<s_1\dots <s_{N}$. As  $\Delta s$ is of order $O(\epsilon^{1-q})$, the following pathwise estimate follows:
\begin{eqnarray*}
&&\left| \Delta s \sum_{n=0}^{N-1}\int_{M_{ H^\epsilon(s_n)} }  g( H^\epsilon(s_n), z) \; d\mu_{  H^\epsilon(s_n)}(z) - \int_{s\wedge \tau^\epsilon}^{(s+t)\wedge \tau^\epsilon}\int_{M_{H^\epsilon(s)} }  g(H^\epsilon(s_n),z) \; d\mu_{H^\epsilon(r)}(z) \;dr\right|\\
&&\le C(g)t \epsilon^{1-q}
\end{eqnarray*}
 where $C(g)=\max_{U_0} |dg|$. 
To summarise:
\begin{equation}
\int_{s\wedge \tau^\epsilon}^{(s+t)\wedge\tau^\epsilon}
 g(y_{r/\epsilon}^\epsilon) dr
=\int_{s\wedge \tau^\epsilon}^{(s+t)\wedge\tau^\epsilon} 
Q^g(H^\epsilon(r)) \;dr+\delta(g,\epsilon,t)
\end{equation}
where 
\begin{eqnarray*}
&&|\delta(g,\epsilon, t)|\\
&\le & C \epsilon^{1-q}+|\epsilon\int_{t_N}^{ {(t+s)\over \epsilon}\wedge \tau^\epsilon} g(y^\epsilon_r) \; dr|
    +| A_1(t, \epsilon)|+\\
   && |A_2(t,\epsilon)-\sum  \epsilon \Delta t \int_{M_{ H^\epsilon(\epsilon t_n)}  }g(H^\epsilon(\epsilon t_n), z) d\mu_{ H^\epsilon(\epsilon t_n)}(z)|\\
&&+ |\sum \Delta s \int_{M_{ H^\epsilon(s_n)}  }g(H^\epsilon(s_n), z) d\mu_{ H^\epsilon(s_n)}(z)-\int_{s\wedge \tau^\epsilon}^{(s+t)\wedge \tau^\epsilon}\int_{M_{H^\epsilon(s)} }  g(H^\epsilon(s), z) \; d\mu_{H^\epsilon(s)}(z) \;ds|
\end{eqnarray*}
and $$A_1(t,\epsilon)=\epsilon \sum_{n=0}^{N-1}\int_{t_n}^{t_{n+1}}\left[ g(y_{r}^\epsilon)- g\left (F_{r-t_n}(y_{t_n}^\epsilon, \Theta_{t_n}(\omega)) \right) \right]\;dr.$$ 
By the previous estimates:
\begin{eqnarray*}
|\delta(g,\epsilon, t)|
&\le& Ct\epsilon^{1-q}+ Ct^{{1\over 2}} \epsilon^{q/2}+|A_1(t, \epsilon)|.
\end{eqnarray*}
To show that $|A_1|$  is reasonably small, we use Kolmogorov's theorem, apply Lemma \ref{lemma-1} and H\"older's inequality 
\begin{eqnarray*}
&&\Big(\E \sup_{s\le t } (A_1(s, \epsilon) )^\beta\Big)^{1\over \beta}\\
&\le&Ct \epsilon^{1-q}+ \epsilon \Big[\E \sup_{s\le t} \Big(\sum_{n=0}^{N-1}\int_{t_n}^{t_{n+1}}\left|g(y_r^\epsilon)- g\left (F_{r-t_n}(y_{t_n}^\epsilon, \Theta_{t_n}(\omega)) \right) \right|\;dr\Big)^\beta\big]^{1\over \beta}\\
&\le& Ct \epsilon^{1-q}+\epsilon\cdot N^{1-1/\beta} \Big(\sum_{n=0}^{N-1}  \E \Big[\sup_{s\le t}
 \int_{t_n}^{t_{n+1}}\left|g(y_{r}^\epsilon)- g\left (F_{r-t_n}(y_{t_n}^\epsilon, \Theta_{t_n}(\omega)) \right) \right|\;dr \Big] ^\beta\Big)^{1\over \beta}\\
&\le& Ct \epsilon^{1-q}+\epsilon N^{1-1/\beta} \cdot N^{1\over \beta}C \cdot \epsilon(\Delta t+(\Delta t)^2)\Delta t \\
&\le& Ct^2 \epsilon^{1-2q}+Ct\epsilon^{1-q}.
\end{eqnarray*}
Consequently,
$$\Big(\E \sup_{s\le t } \Big| \delta(g, \epsilon,s)\Big|^\beta\Big)^{1\over \beta}
\le C t^2 \epsilon^{1-2q}+ Ct^{{1\over 2}} \epsilon^{q/2}+Ct\epsilon^{1-q}$$
and finally take $q=1/4$ to obtain:
\begin{equation}
\Big\| \sup_{s\le t  }  \delta(g, \epsilon,s) \;\Big\|_{L_\beta}
\le C t^2\epsilon^{1\over 4}+C\epsilon^{1\over 4} t^{{1\over 2}}.
\end{equation}
\end{proof}

\begin{theorem}
\label{th-1}
Consider the stochastic differential equation (\ref{sde2}) satisfying condition R.  Let $T^\epsilon$ be the first time that the solution $y_{t\over\epsilon}$ starting from $y_0$ exits  $U_0$. Set $$H_i^\epsilon(t)=H_i(y_{t/\epsilon}^\epsilon).$$
\begin{enumerate}
\item  Let  $\bar H(t)\equiv \bar H_t\equiv  (\bar H_1(t), \dots \bar H_n(t))$ be the solution to the following system of deterministic equations.
\begin{equation}
\label{average-1}
{d\over dt} \bar H_i(t)=\int_{M_{\bar H(t)} } \omega(X_{H_i}, K) (\bar H(t), z) \; d\mu_{\bar H_t}(z),
\end{equation}
with initial condition $\bar H(0)=H(y_0)$. Let $T^0$ be the first time that $\bar H(t)$ exits from $U_0$. Then for all $t<T_0$,  $\beta>1$,
there exists a constant $C_2>0$ such that 
$$\left(\E(\sup_{s\le t} \|H^\epsilon(s\wedge T^\epsilon)-\bar H(s\wedge T^\epsilon)\|^\beta)\right)^{1\over \beta}\le 
C_2\epsilon^{1/4},$$
\item  
Let $r>0$ be such that  $U\equiv \{x: \|H(x)-H(y_0)\|\le r\}\subset U_0$ and define
 $$T_\delta=\inf_t \{|\bar H_t- H(y_0)|\ge r-\delta\}.$$
 Then  for any $\beta>1$,   $\delta>0$ and a constant $C$ depending on $T_\delta$,
 $$P\left(T^\epsilon<T_\delta\right)\le C(T_\delta)\delta^{-\beta} \epsilon^{\beta/4}.$$
\end{enumerate}
\end{theorem}

\begin{remark}
To see that (\ref{average-1}) is a genuine system of ordinary differential equations, take the canonical transformation map $x_a: M_a\to {T^n}$. The pushed forward measure $x_*(\mu_{a})$ is the Lebesque measure $\mu$ on the torus and (\ref{average-1})
becomes:
$${d\over dt} \bar H_i(t)=\int_{T^n} \omega(X_{H_i}, K) \big(x_{\bar H_t}^{-1}(\bar H_t, z)\big) \; d\mu(z).$$
\end{remark}

\begin{proof}
 By It\^o's formula, for $t<T_0\wedge T^\epsilon$. 
$$H_i^\epsilon(t)=H_i(y_0)+ \int_0^t\omega(X_{H_i},K)(y^\epsilon_{s\over \epsilon}) ds.$$
For $i$ fixed, write
\begin{equation}
g_i=\omega(X_{H_i},K)\end{equation}
We only need to estimate
\begin{equation}
\label{diff}
|H^\epsilon_i(t)-\bar H_i(t)|=\Big|\int_0^tg_i(y^\epsilon_{s/\epsilon}) ds-\bar H_i(t)\Big|
\end{equation}
Using the notation of the previous lemma  then
equation (\ref{average-1}) can be written as
\begin{eqnarray*}
{d\over dt} \bar H_i(t)&=& Q^{g_i}(\bar H_t)\\
\bar H_0&=&H(y_0).
\end{eqnarray*}
Apply (\ref{estimate-2}) to the functions $g_i$  we have for any $t<T^\epsilon$,
\begin{eqnarray*}
|H^\epsilon_i(t\wedge T^\epsilon)-\bar H_i(t\wedge T^\epsilon)|
&\le& \int_0^{t\wedge T^\epsilon} |Q^{g_i}(H^\epsilon(s))-Q^{g_i}(\bar H(s))|ds+\delta(g_i,\epsilon, t)\\
&\le & C(g, \phi)\int_0^t\|H^\epsilon(s))-\bar H(s)\|ds+\delta(g_i, \epsilon, t).
\end{eqnarray*}

By Gronwall's inequality, $$\left(\E(\sup_{s\le t\wedge T^\epsilon} \|H^\epsilon(s)-\bar H(s)\|^\beta)\right)^{1\over \beta}\le 
C(t)\epsilon^{1\over 4},$$
concluding part (1) of Theorem \ref{th-1}. 

Part (2) of the theorem easily follows. By definition  $T_\delta$ is the first time that 
$$\sqrt{\sum_i|\bar H_i(s)-H_i(y_0)|^2}\ge r-\delta$$
then
\begin{eqnarray*}
P(T^\epsilon<T_\delta)&\le&
P\left(\sup_{s\le T_\delta\wedge T^\epsilon} \left \|\bar H_s-H^\epsilon(s)\right\| > \delta\right)\\
&\le& \delta^{-\beta} \E\left(\sup_{s\le T_\delta\wedge T^\epsilon} \left\|\bar H_i(s)-H_i^\epsilon(s)\right\| ^\beta\right)\\
&\le&C{\delta^{-\beta} \epsilon^{\beta\over 4}}.
\end{eqnarray*}

\end{proof}

 \section{Perturbation by a Hamiltonian Vector Field}
 \label{se:perturbation}
 If the perturbation $K$ to the stochastic Hamiltonian system (\ref{1}) is a Hamiltonian vector field, \ie $L_X\omega=0$,  then $\int_{M_c} \omega(X_{H_i}, K)d \mu_{c}$ vanishes if $c$ is not a bifurcation value. In fact  let  $(U_0, \phi)$ be an action angle coordinate around $M_c$. We can write $K=X_{k}$ for some smooth function $k$,
\begin{eqnarray*}
 \int_{M_c} \omega(X_{H_i}, K)(z)d \mu_{c}(z)
  &=&\int_{T^n}d\left(k\circ\phi\right)\left(-\sum_{k=1}^n  {\partial (H_k\circ \phi)\over \partial I^k}{\partial \over \partial \theta_k}\right)d \theta\\
  &=&-\sum_{\beta=1}^n \omega^i_\beta(I)\int_{T^n}\left({\partial \over \partial \theta_\beta}\right)\left(k\circ \phi\right)d\theta  =0,\\
\end{eqnarray*}
where $d\theta$ is the standard measure on the n-torus. The ordinary differential equation (\ref{average-1}) governing $\lim_{\epsilon\to 0} H_i(y_{t/\epsilon}^\epsilon)$ has thus a constant solution. In this case we may consider the second order scaling $y_{t/\epsilon^2}^\epsilon$ and the accumulation of the perturbation over a large time interval of order $\epsilon^{-2}$. The proof is inspired by a proof in Hairer-Pavliotis \cite{Hairer-Pavliotis} and  this also benefited  from  the articles by  Khasminski, Papanicolau-Stroock-Varadhan and Freidlin.

Let 
$$\Lo_0(I)={1\over 2}\sum \LL_{X_{H_i}}\LL_{X_{H_i}}+\LL_V$$
be the restriction of the elliptic operator on the invariant manifold with value  $\,I$.  If $f$ on $M_I$ has $\int f d\mu=0$ then the Poisson equation
\begin{equation}
\label{Poisson}
\Lo_0 h=f
\end{equation}
is solvable. We denote by $\Lo^{-1}f$ the solution to the Poisson equation satisfying $\int \Lo^{-1}f d\mu=0$.

Since $\Lo_0$ is elliptic on each level set manifold $M_a$ and  $\{H_i, k\}$ is centered there, the Poisson equation has a unique solution $h_i$.
Note that the functions $\LL_K \{H_i, k\}$ and that $\LL_{X_{H_i}}h_i$ are well defined. 


Note that if $K=X_k$ then the matrix with  $(i,j)$-th entry given by 
$$-dH_i(K) \Lo_0^{-1} \Big( dH_j(K) )\Big)$$ is positive definite.

\begin{theorem}
\label{th-2}
Assume condition $R$ and suppose that $K$ is a smooth local Hamiltonian vector field so that $K=X_{k}$ for some smooth function $k$ in the chart $U_0$. 
Define the matrices $(a_{ij})$  by
\begin{eqnarray*}
a_{ij}(a)&=& -\int_{M_a} \omega(K,X_{H_j})\Lo_0^{-1}\Big(\omega(K, X_{H_i})\Big)(a,z) \; d\mu_a(z) , \quad a\in D\subset \R^n
\end{eqnarray*}
and let $(\sigma_i^j)$ be its square root. Set
$$b_j(a)={1\over 2}\int_{M_a} \LL_K\Lo_0^{-1}( \omega(X_{H_j}, K)) (a, z)  \, d\mu_a(z).$$
Let $z_t$ be the solution to the following stochastic differential equation
$$dz_t^j=\sum_i \sigma^j_i(z_t)\circ dB_t^i+b_j(z_t)dt.$$
Then the law of the stochastic process $H(y^\epsilon_{t\over \epsilon^2})$ stopped at $S^\epsilon$, the first time that the process $y^\epsilon_{t\over \epsilon^2}$ exits from $U_0$, converges to that of $H(z_{t\wedge S^\epsilon})$.
\end{theorem}
{\noindent Remark:} The limiting measure is clearly well defined as $a_{ij}$ and $b_j$ are invariant with different choices of the inverse to $\Lo_0$. 

\begin{proof}
In the following calculation we restrict ourselves on the event $\, \{t<S^\epsilon\}$, equivalently consider the relevant processes stopped at $S^\epsilon$. Set 
\begin{eqnarray*}
\hat y_t^\epsilon&=&y^\epsilon_{{t\over \epsilon^2}\wedge S^\epsilon},\\
\hat H^\epsilon(t)&\equiv& (\hat H^\epsilon_1(t), \dots, \hat H^\epsilon_n(t))=\Big(H_1(\hat y^\epsilon_t), \dots, 
H_n(\hat y^\epsilon_t)\Big).
\end{eqnarray*}
Then 
$$\hat H_i^\epsilon(t)=H_i(y_0)-\epsilon \int_0^{{t\over \epsilon^2} \wedge S^\epsilon}\omega\Big(K, X_{H_i}\Big)(y_s^\epsilon)ds.$$

We first show that the family of the laws $\mu^\epsilon$,  distribution of $\hat H^\epsilon(t\wedge S^\epsilon)$, is tight. It follows, by Prohorov's theorem, that $\{\mu^\epsilon\}$ is relatively compact in the space of probability measures with the topology of weak convergence and hence has a weak limit $\bar \mu$.  To see the tightness of the family $\mu^\epsilon$, we show that for any $a, \eta>0$ there is a $\delta>0$ with 
$$P\Big ( \sup_{|s- t|<\delta}  \|\hat H^\epsilon(t)-\hat H^\epsilon(s)\|^2 \ge a \Big )\le \eta.$$
For this, observe that
$$\|\hat H^\epsilon(t)-\hat H^\epsilon(s)\|^2=\sum_{i=1}^n\Big|-\epsilon \int_{{s\over \epsilon^2} \wedge S^\epsilon}^{{t\over \epsilon^2} \wedge S^\epsilon}\omega\Big(K, X_{H_i}\Big)(y_r^\epsilon)dr\Big|^2.$$
Set $h_i$ to be the solution to the Poisson equation:
$$\Lo_0 h_i=\omega\Big(K, X_{H_i}\Big)$$
with $\int_{M_a} h_0=0$  for any $a\in \R^n$. 
Then 
\begin{eqnarray*}
&&\|\hat H^\epsilon(t)-\hat H^\epsilon(s)\|^2\\
&=&\sum_{i=1}^n\Big|\epsilon \sum_{j=1}^n\int_{{s\over \epsilon^2}  \wedge S^\epsilon}^{{t\over \epsilon^2} \wedge S^\epsilon}
 \LL_{X_{H_j}}h_i(y_r^\epsilon)dB_r^j+\epsilon^2 \int_{{s\over \epsilon^2} \wedge S^\epsilon}^{{t\over \epsilon^2} \wedge S^\epsilon} 
\LL_K h_i(y_r^\epsilon) dr 
-\epsilon h_i(\hat y^\epsilon_t) +\epsilon h_i(\hat y^\epsilon_s)  \Big|^2.
\end{eqnarray*}
Applying Lemma \ref{lemma-2} with $\epsilon$ replaced by $\epsilon^2$, one see that the drift term has a nice bound in $|t-s|$:
 \begin{eqnarray*}
&&\epsilon^2 \int_{{s\over \epsilon^2} \wedge S^\epsilon}^{{t\over \epsilon^2} \wedge S^\epsilon} 
\LL_K h_i(y_r^\epsilon) dr \\
&&=\int_{s\wedge \tau^\epsilon}^{t \wedge \tau^\epsilon} \int_{\hat H^\epsilon(r)} 
\LL_K h_i(z) d \mu(z)dr
+\delta(\LL_Kh_i, \epsilon, t-s),
\end{eqnarray*}
This gives us a comfortable estimates since $\delta(\LL_Kh_i, \epsilon, t-s)$ is of the order $\sqrt \epsilon (t-s)$.
Similarly the quadratic variation of each of martingale terms also converges with the same rate of convergence:
$$\E\Big\langle\int_{{s\over \epsilon^2} \wedge S^\epsilon}^{{t\over \epsilon^2} \wedge S^\epsilon}\LL_{X_{H_j}}h_i(y_r^\epsilon)dB_r^j\Big \rangle
=\epsilon^2 \int_{s\wedge \tau^\epsilon}^{t\wedge \epsilon}
\int \E |\LL_{X_{H_j}} h_i (y_r^\epsilon)|^2 dr.
$$
Applying Burkerholder-Gundy  inequality  to obtain an estimate on the $L_\beta$ norm of  $$\sup_{|s-t|<\delta}\epsilon\Big|\sum_{j=1}^n\int_{{s\over \epsilon^2} \wedge S^\epsilon}^{{t\over \epsilon^2} \wedge S^\epsilon}\LL_{X_{H_j}}h_i(y_r^\epsilon)dB_r^j\Big|,$$
which is a constant multiple of $|s-t|$ plus an error term of the order 
 $\sqrt \epsilon (t-s)$.

Finally it is clear that
$$ \sup_{|s-t|<\delta} | \epsilon h_i(\hat y^\epsilon_t) -\epsilon h_i(\hat y^\epsilon_s)|^2\le C \epsilon \to 0.$$  

To identify the limiting measure let $h$ be the solution to the Poisson equation,
\begin{eqnarray*}
h&=& {1\over 2}\;\sum_{i=1}^n   \partial_i F(H)   \Lo_0^{-1}\Big( \omega(K, X_{H_i})\Big),
\end{eqnarray*}
where $\Lo_0^{-1}$ is considered to act on the angle variable only
and $\int_{M_a}  h =0$ for each $a$. For any smooth function $F$ on $\R^n$, we have 

\begin{equs}[2]
&F\big(\hat H^\epsilon(t))\big)-F(\hat H^\epsilon(0))\\
&=
-\epsilon \sum_{i=1}^n\int_0^{{t\over \epsilon^2} \wedge S^\epsilon} \partial_i F(H(y_s^\epsilon)) \omega(K, X_{H_i})(y_s^\epsilon)ds\\
&= \epsilon\sum_{j=1}^n \int_0^{{t\over \epsilon^2} \wedge S^\epsilon} \LL_{X_{H_j}} h(y_s^\epsilon) dB_s^{j}
+\epsilon^2\int_0^{{t\over \epsilon^2} \wedge S^\epsilon} \LL_Kh(y_s^\epsilon)ds
+\epsilon\Big (h(y_0)-h(\hat y_t^\epsilon)\Big).
\end{equs}
The first term on the right hand side is a martingale and the last term
converges to zero as $\epsilon \to 0$.
We  first identify $\LL_Kh$ in terms of the function $F$. By assumption the functions $\omega(K, X_{H_i})$ are centred  and $\Lo_0^{-1}$ has no effect on functions of $H$ and so
\begin{eqnarray*}
\LL_Kh&=&{1\over 2}\LL_K \Lo_0^{-1}\Big( \sum_{i=1}^n\partial_i F(H) \;\omega(K, X_{H_i})\Big)\\
&=&{1\over 2}\LL_K \Big( \sum_{i=1}^n(\partial_i F)(H) \;\Lo_0^{-1}\big(\omega(K, X_{H_i})\big)
\Big)\\
&=&-{1\over 2}\sum_{i=1}^n \partial_j \partial_i F(H)\omega(K,X_{H_j})\Lo_0^{-1}\Big(\omega(K, X_{H_i})\Big)\\
&&+{1\over 2} \Big( \sum_{i=1}^n\partial_i F(H) \LL_K\Lo_0^{-1}( \omega(X_{H_i},K)) \Big).
\end{eqnarray*}
Set
$$\bar \Lo=-{1\over 2}\sum_{i,j}\omega(K,X_{H_j})\Lo_0^{-1}\Big(\omega(K, X_{H_i})\Big) \partial _i \partial_j+{1\over 2} \sum_{i=1}^n \LL_K\Lo_0^{-1}( \omega(K, X_{H_i}))\partial_i,  $$
to see 
\begin{eqnarray*}
&&F\big(\hat H^\epsilon(t))\big)-F(\hat H^\epsilon(0))\\
&=&  \epsilon\sum_{j=1}^n \int_0^{{t\over \epsilon^2}\wedge S^\epsilon}  \LL_{X_{H_j}} h(y^\epsilon_s) dB_s^{j}
+\epsilon^2\int_0^{{t\over \epsilon^2}\wedge S^\epsilon} \bar \Lo F \circ H(y_s^\epsilon)ds
+\epsilon\Big (h(y_0)-h(\hat y_t^\epsilon)\Big).
\end{eqnarray*}
Mimicking Papanicolaou-Stroock-Varadhan, we define  $\F^H_s\equiv \sigma\{\hat H^\epsilon_{r\wedge S^\epsilon}: r\le s\}$ and so $\{\F_s^H: s\ge 0\}$ is the filtration generated by $\hat H^\epsilon_{\cdot\wedge S^\epsilon}$. We need the following estimates: 
\begin{eqnarray*}
&&\epsilon^2\int_{a\wedge T^\epsilon}^{{t\over \epsilon^2}\wedge S^\epsilon}  \bar \Lo F(  H(y_s^\epsilon))ds\\
&=&\int_{a\wedge T^\epsilon}^{t\wedge T^\epsilon} \Big(\int_{M_{\hat H^\epsilon (s)}} \bar \Lo F \circ  H(z)\; d\mu_{\hat H^\epsilon (s)}(z)\Big) ds+ \delta(\bar \Lo F, \epsilon^2, t-a)\\
&=&\int_{a\wedge T^\epsilon}^{t\wedge T^\epsilon}   \Lo F\circ \hat H^\epsilon(s)\; ds+ \delta(\bar \Lo F\circ  H, \epsilon^2, t-a),
\end{eqnarray*}
where in the action-angle local coordinate, 
\begin{eqnarray*}
 \Lo  F(a)&=&-{1\over 2}\sum_{i,j=1}^n\partial_j \partial_i F(a)\;\int_{M_a} \omega(K,X_{H_j})\Lo_0^{-1}\Big(\omega(K, X_{H_i})\Big)(a,z) \; d\mu_I(z)  \\
&&+{1\over 2}\sum_{i=1}^n\partial_i F(a) \int_{M_a} \LL_K\Lo_0^{-1}( \omega(X_{H_i}),K)(a, z)  \, d\mu(z).
\end{eqnarray*}
Then for  any $\F^H_s$- measurable $L^2$ random function $G$, using again Lemma \ref{lemma-2},
\begin{eqnarray*}
&&\E G\; \Big[ F\big(\hat H^\epsilon(t\wedge S^\epsilon))\big)-F\big(\hat H^\epsilon(s\wedge S^\epsilon))\big)
-\int_{s\wedge S^\epsilon}^{t\wedge S^\epsilon} \int_{M_{\hat H^\epsilon (r)}} (\bar \Lo F)(z)d\mu_{\hat H^\epsilon (r)}(z) dr\Big]\\
&=&\E G \Big[  \epsilon\sum_{j=1}^n \int_{{s\over \epsilon^2}\wedge S^\epsilon} ^{{t\over \epsilon^2}\wedge S^\epsilon} \LL_{X_{H_j}} h(y_s^\epsilon) dB_s^{j} \Big]\\
&&+\E G \; \Big[ \delta(\bar \Lo F, \epsilon^2, t-s)  +\epsilon\Big (h(y_{{s\over \epsilon^2}\wedge S^\epsilon} ^\epsilon)-h(y_{{t\over \epsilon^2}\wedge S^\epsilon}^\epsilon)\Big)\Big]\\
&=&\E G \; \Big[ \delta(\bar \Lo F, \epsilon^2, t-s)+\epsilon\Big (h(y_{{s\over \epsilon^2}\wedge S^\epsilon}^\epsilon)-h(y_{{t\over \epsilon^2}\wedge S^\epsilon}^\epsilon)\Big)\Big]\to 0.
\end{eqnarray*}
Consequently $$\E \Big\{ F\big(\hat H^\epsilon(t\wedge S^\epsilon))\big)-F\big(\hat H^\epsilon(s\wedge S^\epsilon))\big)
-\int_{s\wedge S^\epsilon}^{t\wedge S^\epsilon} \int_{M_{\hat H^\epsilon (r)}} (\bar \Lo F)(z)d\mu_{\hat H^\epsilon (r)}(z) dr  \big | \F_s^H\Big\}
\to 0,$$
and so any weak limit of the law $\hat H^\epsilon_\cdot$ is the solution to the  martingale problem for the second order differential operator $\Lo$.

\end{proof}

\bigskip

{\bf Acknowledgement.} 
{This research has benefitted from inspiring discussions with Martin Hairer. I would also like to thank R. Hudson and the referees for critical reading and comments. }

\end{document}